\documentclass[10pt]{amsart}

\usepackage{amsmath}
\usepackage{amsfonts}
\usepackage{amssymb}
\usepackage{amsthm}
\usepackage{graphicx}
\usepackage{latexsym}
\usepackage{float}
\usepackage{url}
\usepackage[all]{xy}
\usepackage{setspace}
\usepackage{eucal}
\usepackage{enumerate}
\usepackage{verbatim}
\usepackage{indentfirst}
\usepackage{multirow}

\theoremstyle{plain}

\newtheorem{thm}{Theorem} [subsection]
\newtheorem{prop}[thm]{Proposition}
\newtheorem{lema}[thm] {Lemma}
\newtheorem{cor}[thm]{Corollary} 
\newtheorem*{mythm}{Theorem} 
\theoremstyle{definition}

\newtheorem{defn}[thm]{Definition}
\newtheorem{exmple}[thm]{Example}
\newtheorem{rem}[thm]{Remark} 
\newtheorem{notn}[thm]{Notation}

\newcommand*{\C}{\ensuremath{\mathbb{C }}}         
\newcommand*{\R}{\ensuremath{\mathbb{R }}}                        
\newcommand*{\Z}{\ensuremath{\mathbb{Z }}}                        
\newcommand*{\Zlat}[1] {\ensuremath{\mathbb{Z }^{#1} \cup \{ \infty\}}} 
\newcommand*{\Q}{\ensuremath{\mathbb{Q }}}                       
\newcommand*{\res}{\ensuremath{\mathbf{k}}}  
\newcommand*{\K}{\ensuremath{\mathbb{K }} }  
\newcommand*{ \Kn}[1] {\ensuremath{\mathbb{K }_{(#1)}}} 
         
\newcommand*{\Lf}{\ensuremath{\mathbb{L }} }  
\newcommand*{ \Ln}[1] {\ensuremath{\mathbb{L}_{(#1)}}} 
\newcommand*{ \no}[1] {\ensuremath{(#1)^{\times}}} 
\newcommand*{ \nob}[1] {\ensuremath{#1^{\times}}} 
\newcommand*{\Sc}[1]{\ensuremath{\mathcal{S}_{#1}} }  
\newcommand*{\eqdef}{\stackrel{\text{def}}{=}}     
\newcommand*{\X}{\ensuremath{\mathcal{X}}} 
\newcommand*{\I} {\ensuremath{\mathcal{I}}} 
\newcommand*{\Rint}[1]{\ensuremath{\mathcal{O}_{#1}}} 
\newcommand*{\mideal}[1]{\ensuremath{\mathfrak{m}_{#1}}} 
\newcommand*{\gm} {\ensuremath{\Gamma}} 
\newcommand*{\gmext}{\ensuremath{\Gamma \cup \{ \infty\} }} 
\newcommand*{\xbar}[1] {\ensuremath{\overline{#1}}} 

\newcommand*{\tor} {\ensuremath{\mathcal{T}}}
\newcommand*{\ndim}[1] {\textit{#1}-dimensional}

\newcommand*{\mbf}[1]{\ensuremath{\mathbf{#1}}}
\newcommand*{\bfvec}[1]{\ensuremath{(#1_{1}, \ldots, #1_{n})}}

\newcommand*{\ben} {\begin{enumerate}}
\newcommand*{\een} {\end{enumerate}}
\newcommand*{\bit} {\begin{itemize}}
\newcommand*{\eit} {\end{itemize}}
\newcommand*{\noarch}{non-Archimedean}

\begin{document}
\title{Tropical geometry over higher dimensional local fields}
\author{S. Banerjee}
\address{Department of Mathematics\\
Yale University\\ 
New Haven, CT 06511, USA}
\email{soumya.banerjee@yale.edu}

\keywords{tropical geometry, higher local fields}
\subjclass[2010]{Primary : 14T05; Secondary : 52A20 }
\date{\today }

\begin{abstract} 
We introduce the tropicalization of closed subschemes of a torus defined over a higher dimensional local field. We study the basic invariants of such tropicalizations. This is a generalization of the results of Einslieder, Kapranov and Lind, \cite{EKL}, Speyer \cite{Sp} and Speyer and Sturmfels \cite{Ss} to higher dimensional local fields. 
\end{abstract}
\maketitle

\section{Introduction}

Let $\res$ be a local field and $\res^{al}$ be a fixed algebraic closure. The valuation on $\res$ extends to a group homomorphism $\nu : (\res^{al})^{\times} \rightarrow \Q.$ If $\X$ is any subvariety of the $m$- dimensional algebraic torus $\tor_{\res} = \mathbb{G}_{m, \res}^{m},$ one defines the tropical  variety associated to $\X$, denoted by $Trop(\X),$ as the closure of the set $\{ \nu(\mbf{z}) : \mbf{z} \in \X \} \subset \R^{m}.$  Such sets were studied by Bieri and Groves, \cite{Bieri-Groves} and later by Einslieder, Kapranov and Lind \cite{EKL}, Speyer \cite{ Sp}; Speyer and Sturmfels \cite{Ss}. More recently,  Payne \cite{Pay1}, has considered further generalizations of tropicalization (the map $Trop$)  to toric varieties.

Higher dimensional local fields are a generalization of the concept of local fields. They were introduced by Parshin \cite{Parshin} and Kato \cite{Kato3, Kato2, Kato1} to generalize class field theory to higher dimensional schemes. Informally speaking, a \ndim{n} local field is a sequence of discretely valued complete fields $(\Kn{0}, \Kn{1}, \ldots \Kn{n})$ such that each $\Kn{i}$ is the residue field of $\Kn{i+1}.$ For instance, $\mathbb{F}_{p}((t_{1}))((t_{2}))\ldots ((t_{n}))$ is a \ndim{n} local field. \\
Geometric construction of such a field was given by Parshin. The idea is to look at the completion of  function fields associated to a fixed flag of smooth closed sub-schemes $\Sc{0} \subset \ldots \Sc{n-1}\subset \Sc{n}$ where dim$(\Sc{i}) = i.$ We refer the reader to \cite{Parshin} for details.

In case of a \ndim{n} local field $\K,$ the value group is a rank \textit{n} free abelian group. As usual, the valuation group comes with an ordering. It turns out to be the lexicographic ordering on $\Z^{n}.$ This means that the components of the valuation are asymmetric in nature. It  turns out that  extension of this map to the algebraic closure $ \K^{al}$ is not as straightforward as in the rank one (local field) case and some care is needed to define it.

The recent interest in Okounkov bodies, as exemplified in the work of Mustata and Lazarsfeld \cite{ML}, have demonstrated the importance of such vector valued valuations. Motivated by their work, we found out that the theory of tropical varieties over \noarch fields generalizes to higher local fields.

 Let $\K$ be a \ndim{n} local field. Generalizing the one dimensional valuation map we have a vector valued valuation $\nu :\no{ \K^{al}} \rightarrow \Gamma_{\R}.$ where $\Gamma_{\R} \cong \R^{n}$ as an abelian group with lexicographic order.
Now if $\tor_{\K}$ is an \ndim{m} algebraic torus and $M$ is the lattice of characters then we define $Trop$ as a map from $\tor_{\K}(\K^{al})$ to $\mbox{Hom}(M, \Gamma_{\R}),$ which is ``dual" to the co-ordinate wise valuation, see  definition \ref{D:trop} for a precise statement.
The main result of this paper, Theorem \ref{T:main2},  is the following:
\begin{mythm} 
The tropicalization of a $d$- dimensional closed, reduced, irreducible subscheme of an algebraic torus over a \ndim{n} local field is a rational polyhedron complex of dimension $nd.$
\end{mythm}

In the case when $\X$ is a hypersurface, we provide an independent geometric proof of the above theorem by introducing the \textit{n-extended Newton polytope}. This generalization is motivated by the \textit{extended Newton polyhedron} construction of the authors of \cite[Section 2]{EKL}. Tropicalization of a hypersurface, defined over a field with arbitrary value group, has been considered by Aroca in \cite{Aroca}. In the case of higher local fields, our results are obtained (independently) as a special case of \cite{Aroca}. 

We give an example of a hypersurface (Example \ref{E:impure}), which shows that polyherdral complexes obtained as tropicalization over higher dimensional fields are not necessarily pure. i.e. the maximal faces (with respect to inclusion) need not have the same dimension.   

We have not addressed the question of  connectedness of the tropicalization map in this paper. In case of a local field, one standard approach is to view the tropicalization of a variety, as a suitable projection of the  Berkovich analytification of the same variety, see \cite{Pay2}. However, the corresponding theory of analytification associated to a higher rank valuations is not developed. We hope that our work will be helpful to extend the theory of rigid analytic spaces to higher local fields, in the spirit of \cite{Pay2}.

In their fundamental work, Speyer and Sturmfels pioneered the use of Gr\"{o}bner basis techniques to study tropicalization. It is possible to interpret the higher rank value groups, that arise in our context, as generalized weight vectors on polynomial rings. It then seems possible to extend the results of Gr\"{o}bner basis theory to higher local fields.

\subsection*{Acknowledgement}
I am deeply grateful to my advisor M.M.Kapranov for introducing me to this problem, many useful discussions and suggesting key improvements in the exposition and corrections. It is a pleasure to thank Sam Payne for many enlightening discussions on tropical geometry.

\tableofcontents
\section{Valuations}\label{Sec:Val}
In this section we recall some well known results in valuation theory. We are interested in the (special) case of valuations which are discrete and  of finite rank. Our main reference is \cite{Enp}.
\subsection{Ordered abelian groups} 
\begin{defn}
An ordered abelian group $(\gm, +, \leq),$ is an abelian group $(\gm, +)$ such that the underlying set $\gm$ is a totally ordered and addition preserves the ordering.
\end{defn}
A subgroup $\Delta \subset \gm$ is called convex if for any $\gamma \in \gm$  and $ 0 \leq \gamma \leq \delta \in \Delta \Rightarrow \gamma \in \Delta .$ It is clear that the set of all convex subgroups is linearly ordered by inclusion. The maximal length of a chain of distinct convex proper subgroups is called the \textbf{rank} of the ordered abelian group.\\
In particular the only proper convex subgroups of a rank 1 abelian group is the trivial subgroup $\{0\}.$
\begin{exmple} 
\ben [(a)]
\item $(\Z,+) , (\Q, +)$  are rank \textit{1} abelian groups with their natural ordering induced from $\R.$
\item \label{I:one}
If $(\Gamma_{1}, \preceq_{1})$ and $(\Gamma_{2}, \preceq_{2})$ are two ordered abelian groups, then we can define an ordering, called the lexicographic order on $\gm \eqdef \Gamma_{1} \oplus \Gamma_{2}.$ It is defined as follows:
\[ (\gamma_{1}, \gamma_{2}) \preceq (\gamma^{\prime}_{1}, \gamma^{\prime}_{2}) \text{ if   either }
\gamma_{2} \preceq_{2} \gamma^{\prime}_{2} \text{ or } \gamma_{2} =_{2} \gamma^{\prime}_{2}  \text{ and }
\gamma_{1} \preceq_{1} \gamma^{\prime}_{1}   \] 
 Clearly, $\text{rank}(\gm, \preceq) =  \text{rank}(\gm_{1}, \preceq_{1}) + \text{rank}(\gm_{2}, \preceq_{2}).$\\ 
In general $(\Z^{n}, +)$ with its lexicographic order has rank \textit{n}. 

\een
\end{exmple}
An ordering $\preceq$ of $\gm$ is called \textbf{discrete} if there is a unique minimal positive ($ \succeq 0 $) element. 
An ordering $\preceq$ of $ \gm$ is called \textbf{Archimedean} if for every $\alpha, \beta \in \gm$ there is a positive integer  $n$ such that $\beta \preceq n\alpha.$ It is clear that any abelian group with an Archimedean ordering is of rank 1. The following proposition characterizes abelian ordered groups of rank 1.
\begin{prop}[Proposition 2.1.1 \cite{Enp}]
An ordered abelian group $\gm$ is of rank 1 iff it is order isomorphic to a non-trivial subgroup of $(\R, +)$ with the canonical order induced from $\R.$ 
\end{prop} 
\begin{cor}\label{C:disc}
Any abelian group with a discrete rank \textit{1} ordering is order isomorphic to $\Z$ .
\end{cor}
\begin{rem}Given any ordered abelian group $\gm,$ we will augment it by a symbol $\infty$ with the properties $ \gamma \preceq \infty$ and $\gamma + \infty = \infty$ for all $\gamma \in \gm .$ We will denote the enlarged set by $\gmext.$ 
\end{rem}
\begin{defn}
A valuation $\nu$ on a field $\res$ is a \textbf{surjective} map $\nu: \res \rightarrow \gmext $ satisfying the following properties:
\begin{eqnarray} \label{P:val}
\begin{array}{c}
\nu(0) \eqdef \infty  \\
\nu(ab) = \nu(a) + \nu(b) \\
\nu(a+b) \geq \min \{\nu(a), \nu(b)\}
\end{array}
\end{eqnarray}
 \end{defn}
 \begin{rem} 
 The rank of a valuation is defined to be the rank of the abelian group $\nu(\res^{\times}).$ If $\gm= \{0\}$ then $\nu$ is called a trivial valuation. 
 \end{rem}
 Associated to a valuation $\nu: \res \rightarrow \gmext $ is the ring $\Rint{\nu} = \{ x \in \res : \nu(x) \geq 0  \}.$ 
 A ring $R \subset \res$ is called a valuation ring if for all $x \in \res$  either $x \in R$ or $x^{-1} \in R.$ The ring $\Rint{\nu}$ is a canonical valuation ring associated to $\nu.$ Conversely we have :
 \begin{prop}[Proposition 2.1.2, \cite{Enp}]\label{P:valring}
 Given any valuation ring $R \subset \res.$ There is an ordering on the abelian group $\res^{\times}/R^{\times}$ such that the  quotient map  $\nu : \res ^{\times}\rightarrow \res^{\times}/R^{\times}$ is a valuation and $R = \Rint{\nu}$
 \end{prop}
 The group $\displaystyle \Rint{\nu}^{\times}$ of units in $\displaystyle \Rint{\nu}$ is given by \linebreak[3] $\displaystyle \left \{ x \in \res : \nu (x) = 0  \right \}.$ The ideal $\displaystyle \mideal{\nu} = \{ x \in \res : \nu (x) \ > 0  \}$ is a maximal ideal of $\displaystyle \Rint{\nu}.$ The field $\displaystyle \xbar{\res} \eqdef \Rint{\nu}/ \mideal{\nu} $ is called the \emph{residue field}. It is to be noted that the ring $\displaystyle \Rint{\nu}$ is not in general, a local ring, for arbitrary valuations. 
 \begin{defn}
 Two valuations $v_{i} : \res \rightarrow \gm_{i} \cup \{ \infty \},$ $(i = 1, 2)$ are called equivalent if they define the same valuation ring i.e  $\Rint{\nu_{1}} = \Rint{\nu_{2}}.$
 \end{defn}
 \begin{prop}[Proposition 2.1.3 \cite{Enp}]
 Two valuations $\nu_{i} : \res \rightarrow \gm_{i} \cup \{ \infty \} \; (i=1,2)$  are equivalent iff there is an order-preserving isomorphism 
 $\iota : \gm_{1} \rightarrow \gm_{2} $ such that $\iota \circ \nu_{1} = \nu_{2}$
 \end{prop}
 \subsection{Extension of valuations} \label{Ss:extval}
If $\res_{2}$ is an extension of $\res_{1}$ and $\nu_{2} : \res_{2}^{\times} \rightarrow \gm_{2}$  and $\nu_{1} : \res_{1}^{\times} \rightarrow \gm_{1}$ are given valuations, we say that $\nu_{2}$ is an extension of $\nu_{1}$ if the following diagram commutes 
\begin{displaymath}
    \xymatrix{
        \res_{1}^{\times} \ar[r]^{i}\ar[d]^{\nu_{1}} &  \res_{2}^{\times}  \ar[d]_{\nu_{2}} \\
        \gm_{1} \ar[r]^{f}       & \gm_{2}}
\end{displaymath}
where $f: \gm_{1} \rightarrow \gm_{2}$ is an order preserving injective group homomorphism. It turns out from a theorem of Chevalley that valuation can always be extended to any field extension. 
\begin{thm}[Chevalley, Theorem 3.1.1, \cite{Enp}] 
For a field $\res,$ let $\mbf{R} \subset \res$ be a subring and $\mathfrak{p} \subset \mbf{R}$ be a prime ideal of $\mbf{R}.$ Then there is a valuation ring $\Rint{}$ of $\res$  and a maximal ideal $\mideal{\Rint{}}\subset \Rint{}$ such that  $\mbf{R} \subset \Rint{}$ and $\mideal{\Rint{}} \cap  \mbf{R} = \mathfrak{p}.$
\end{thm}
\begin{cor} Let $\res_{2} / \res_{1}$ be any extension of fields, let $\Rint{1} \subset \res_{1}$ be any valuation ring. Then there is a valuation ring $\Rint{2} \subset \res_{2}$ such that $\Rint{2} \cap \res_{1} = \Rint{1} .$
\end{cor}
When dealing with algebraic extensions one has a little more information on the rank of the extension of the valuation. More precisely, we have:
\begin{thm}[Theorem 3.2.4, \cite{Enp}]
Suppose that $\res_{2}$ is an algebraic extension of the valued field $\res_{1}$ and $\Rint{2}$ is a valuation ring extending the valuation ring $\Rint{1} \subset \res_{1}.$  Then the following holds:
\ben
\item The group $\gm_{2} / \gm_{1}$ is a torsion group where $\gm_{i} = \res_{i}^{\times}/ \Rint{i}^{\times} \; (i=1,2).$
\item $\xbar{\res_{2}}$ is an algebraic extension of $\xbar{\res_{1}}.$
\een
\end{thm}
\begin{cor}
The rank of extended valuation remains the same for algebraic extensions of fields.
\end{cor}


\section{Higher dimensional local  fields} \label{Sec:Higherlocalfields}
\begin{defn}\label{D:hloc} A \ndim{n} local field over a field $\res$ is an ordered sequence of fields, \linebreak[3] $(\Kn{0}, \Kn{1}, \ldots, \Kn{n-1}, \Kn{n})$  such that :
\ben [1.]
\item Each $\Kn{i},\; 1 \leq i \leq n $ is a complete with respect to a discrete rank one valuation $w_{(i)} :\nob{ \Kn{i}} \rightarrow \Z $ (i.e. each $\Kn{i}$ is a local field). 
\item $\Kn{i}$ is the residue field of $\Kn{i+1}$ for all $i \geq 0.$
\item $\Kn{0} = \res$

\een
\end{defn} 
\begin{rem}
\ben[(i)] \label{C:convention}
\item We will use the phrase ``$\K$ is a \ndim{n} local field" when the underlying sequence of fields  \linebreak[2] $ (\Kn{0}, \ldots, \Kn{n-1}, \Kn{n})$ is clear from the context and $\K = \Kn{n}.$ 
\item The valuation ring of $\Kn{i}$ with respect to $w_{(i)}$ will be always denoted by $\Rint{\Kn{i}}.$ The maximal ideal of  $\Rint{\Kn{i}}$ is denoted by $\mideal{\Kn{i}}.$ 
\item Henceforth, we will always assume that each $w_{(i)}$ is a nontrivial (and \noarch) valuation (i.e. $w_{(i)} \not\equiv 0$).
\een
\end{rem}
\begin{exmple}[Equal Characteristic] 
Given any field $\res,$ consider $\Kn{i} \eqdef \linebreak[4] \res((t_{1})) \ldots ((t_{i})),$ complete with respect to the valuation $ord_{t_{i}}: \nob{\Kn{i}} \rightarrow \Z.$ Then $\Kn{n}$ is a \ndim{n} local field over $\res.$ 
\end{exmple}
 \begin{exmple}[Unequal Characteristic]
If $\K$ is any field complete with respect to a discrete valuation $w: \K \longrightarrow \Zlat{} $ consider the field
 \[
  \K  \{ t \} \eqdef  \left\{ \sum_{ i = - \infty }^{\infty} a_{i} t^{i} : a_{i} \in \K ,\; \inf_{i \in \Z} \{w(a_{i})\} \gg -\infty \; \mbox{ and }  w(a_{i}) \rightarrow \infty \; \mbox{ as } i \rightarrow -\infty \right \}
\]
Our assumptions on the growth conditions of the coefficients make the algebraic addition and product 
(\emph{usual power series operations}) are well defined.

Now define a discrete valuation $\nu$ on $ \K  \{ t \}$ (extending $w$) as follows :
 \[\nu(\sum_{ i = - \infty }^{\infty} a_{i} t^{i}) = \inf_{i \in \Z}\{ w(a_{i})\}\]
  Then  $ \K  \{ t \}$  is complete with respect to the valuation $\nu.$  
  
If moreover $\res $ is the residue field of $\K$ (with respect to the valuation $w$) then the residue field of $ \K  \{ t \}$ is $\res((t)).$ Note that char($\res$) = char($\res((t))$) and in general it is not equal to char($\K\{ t\}$). 
\end{exmple}
The structure theorem for higher local fields is as follows:
 \begin{thm}[\cite{Zhukov2}]
 Let $\K $ be an \ndim{n} local field. In case, char$(\K) = 0$ and char$(\Kn{0}) = p$, denote by $k_{0}\hookrightarrow \K $, the quotient field of $W(\Kn{0})$ (the Witt ring of $\Kn{0}$, see \cite{Serre}). Then $\K$ is one of the following fields.
 \ben[(a)]
 \item If  char($\K) = \text{ char}(\Kn{0}) $, then $\K$  is isomorphic to $\Kn{0}((t_{1}))\ldots((t_{n})).$
 \item If  char$(\Kn{m}) = p$ and char$(\Kn{m+1}) =0$, $ m \geq 1$, Then $\K$ is a finite totally ramified extension of the the field $k\{\{t_{1}\}\}\ldots \{\{t_{m}\}\}((t_{m+1})) \ldots ((t_{n}))$ where $k$ is a finite extension of $k_{0}.$
 \item If  char$(\Kn{1}) = 0 $ and char$(\Kn{0}) = p$ then $\K$ is isomorphic to $ k((t_{1}))\ldots((t_{n-1}))$ where $ k$ is a finite extension of $k_{0}.$
 
 \een
 \end{thm}
 \subsection{Algebraic Extensions}
 Let $\K$ be an \ndim{n} local field and $\Lf$ be a finite field extension of $\K$. Recall by the conventions in \ref{C:convention}, we will interchangeably use $\K$ and $\Kn{n}.$   By \cite[Ch2, Prop 3]{Serre}, $\Lf$ is a complete with respect to a (non-trivial) discrete valuation, $w^{\Lf}_{n}.$ Moreover from the same theorem, we have the valuation ring  of $\Lf$ with respect to $w^{\Lf}_{n}$, denoted by $\Rint{\Lf}$ is a free module of finite rank over $\Rint{\Kn{n}}.$ As a result the residue field of $\Lf $ denoted by $\Ln{n-1}$ is again finite dimensional field extension of $\Kn{n-1}.$\\
 Using the aforementioned theorem repeatedly, we see that $\Lf$ itself becomes a \ndim{n} local field with each $\Ln{i}$ a finite extension of $\Kn{i}.$ We say that the structure of $\Lf$ as a \ndim{n} local field is compatible with that of $\K.$
 \subsection {A rank n valuation on $\K$} 
 Let us consider the following subsets of a \ndim{n} local field $\K$
\begin{eqnarray}
 P & = &  \lbrace f \in \K : w_{(n)}(f) \geq 1 \rbrace \nonumber \\
 Q_{i} & = & \lbrace f \in \K : w_{(n)} (f) = \ldots = w_{(n-i+1)} (f) = 0 \mbox{ and } w_{(n-i)}(f) \geq 1   \rbrace \nonumber \\
 & & \mbox{for }  1 \leq i  \leq n-1 \nonumber \\
 R & = & \lbrace f \in \K :  w_{(n)} (f) = \ldots = w_{(n-i+1)} (f) = w_{(1)}(f) = 0 \rbrace \nonumber \\
 \text{ Let} \; \Rint{\nu}^{\K} & \eqdef  & P \cup Q_{i} \cup R \nonumber
   \end{eqnarray} 
Then $\Rint{\nu}^{\K}$ is a commutative ring with one. 
\begin{lema}
$\Rint{\nu}^{\K}$ is in fact a valuation ring (see section \ref{Sec:Val} for definition).
Let  $\displaystyle \Gamma^{\K} \eqdef \nob{\K}/ \no{\Rint{\nu}^{\K}}.$ 
Then $ \displaystyle \Gamma^{\K}   \cong \Z^{n}.$ Here $\no{\Rint{\nu}^{\K}}$ is the set of all elements in $\Rint{\nu}^{\K}$ whose inverses also belong to the same ring. 
\end{lema}
 \begin{proof}
 If $f \in \nob{\K}$ then (replacing $f$ with $f^{-1}$ if  necessary) we always have $w_{(n)} (f) = \ldots = w_{(n-i+1)} (f)= 0 $ and $w_{(n-i)}(f) \geq 0 $ as each $\Rint{\Kn{i}}$ is a valuation ring. This proves the first assertion.

 Now, consider the following descending filtration on $\nob{\K}$  
 \begin{eqnarray*}
 S^{0}(\K) & = & \K^{\times}  \nonumber \\
 S^{1}(\K)  &= & \lbrace f \in \K^{\times} : w_{(n)} (f)  \rbrace \nonumber \\
 S^{2}(\K)  &= & \lbrace f \in \K^{\times} : w_{(n)} (f) = w_{(n-1)} (f)= 0   \rbrace \nonumber \\
 & \vdots &  \nonumber \\
S^{n}(\K)  &= & \lbrace f \in \K^{\times} : w_{(n)} (f) = \ldots = w_{(1)} (f)= 0   \rbrace . \nonumber \\
 \end{eqnarray*}
 
 The induced filtration  on $ \Gamma^{\K}$, is given by 
 \[  F^{i}(\Gamma^{\K})  = S^{i} (\K) / (S^{i}(\K) \cap  \no{\Rint{\nu}^{\K}}).
\]
 We then have $F^{i}(\K) / F^{i+1}(\K) \cong \Z.$
 
 Let us consider now an increasing filtration $G_{i}(\Gamma^{\K})$ defined by
 \[ G_{i}(\Gamma^{\K}) = F^{n-i}(\Gamma^{\K}). \]
 We now pass to the associated graded construction and we have the required isomorphism  $G_{i}(\Gamma^{\K}) \cong \Z^{n}.$
 \end{proof}
 
 \begin{cor}
 If we totally order $\Gamma^{\K}$ by transporting the lexicographic order from $\Z^{n}$, using the above isomorphism (with respect to the filtration $G_{\bullet}$),  then $\displaystyle \nu: \nob{\K}\rightarrow  \Gamma^{\K}$ is a valuation map.
 \end{cor}
 \begin{rem}
In literature, the valuation map is concretely realized by using the concept of local system of parameters.  Let $\pi_{i} $ be any uniformizing element of $\Kn{i}$ (i.e. $w_{(i)}(\pi_{i}) = 1$). \\
A set $\mbf{t}_{\K} \eqdef (t_{1}, t_{2}, \ldots, t_{n})$ where $t_{n} = \pi_{n}$ and $t_{i} \in \no{\Rint{\Kn{n}}}$ is an arbitrary lift of $\pi_{i}$ for $i \leq n$ is called a \textbf{local system of parameters} of $\K.$\\
The valuation  is then defined in terms of successive reduction to the residue field by using these local parameters, see   \cite{Zhukov1, Zhukov11} for details.
 \end{rem} 
  If  $\Lf$ is any finite algebraic extension of $\K$, then $\Rint{\nu}^{\Lf}$ is the valuation ring over $\Rint{\nu}^{\K}.$ So we have the following commutative diagram
 \[
  \xymatrix{
  \nob{\K} \ar[d] \ar[r]^{\nu} & \Gamma ^{\K} \ar[d] \\
  \nob{\Lf} \ar[r]_{\nu} & \Gamma^{\Lf}
   }
  \] 
  
 Moreover the filtration $G_{\bullet}$ considered above has the property that $G_{i}(\Lf)/ G_{i}(\K)$ is finite abelian group for all $i \geq 0.$
 Passing to $\K^{al}$, the algebraic closure of $\K$, we get a map $\nu : \no{\K^{al}} \rightarrow  \Gamma^{\K}_{\Q}$ , where $\Gamma^{\K} _{\Q} \eqdef \varinjlim{\Gamma^{\Lf}}$ over all finite extensions. 
 
 It is well known, that $\Gamma^{\K} _{\Q}$ is divisible and hence a $\Q$ module. It contains $\Gamma^{\K}$ as a $\Z$ lattice. Let  $\Gamma^{\K} _{\R} \eqdef \Gamma^{\K} _{\Q} \otimes_{\Q} \R.$ It is an \ndim{n} vector space with a given isomorphism with $\R^{n}.$ Under this isomorphism, the lexicographic order on $\R^{n}$ induces a total ordering on $\Gamma^{\K} _{\R}.$
 
We will refer to the map $\displaystyle \nu :\no{\K^{al}} \rightarrow \Gamma^{\K}_{\R}$, and its restriction to sub-fields  as the valuation map. When the underlying field is clear from context, it will almost always be denoted by $\K$, we will for brevity, write $\Gamma_{\R}$ instead of $\Gamma^{\K}_{\R}.$
\section{Newton Polytopes} \label{Newton}
\begin{notn} 
\ben[(a)]
\item For any $\mbf{d} \in \Z^{n}$ we define a monomial $\mbf{x}^{\mbf{d}} \eqdef \prod_{i} x_{i}^{d_{i}}.$ 
\item For any field $\res$ and a polynomial  $f(\mbf{x}) = \sum_{\Z^{n}}\; a_{\mbf{d}}\mbf{x}^{\mbf{d}},$ the support of $f$ is defined as the set  Supp($f$)$ \eqdef \{ \mbf{d} \in \Z^{n} :  a_{\mbf{d}} \neq 0\}$ .
\item For any subset $S$ of a real vector space we denote the convex hull of $S$ by Conv$\{S \}.$

\item  If $\mbf{x} = (x_{0}, \ldots, x_{n})$ and $\mbf{y} = (y_{0}, \ldots, y_{n})$ denote two points in $\R^{n+1}$ such that $x_{0} \neq y_{0}$ we define the generalized slope of the vector $\overrightarrow{\mbf{xy}}$  as the vector 
\[m(\overrightarrow{\mbf{xy}}) \eqdef (y_{1}- x_{1}/ y_{0} - x_{0}, \ldots, y_{n}- x_{n}/ y_{0} - x_{0})  \in \R^{n}\]
\een
\end{notn}
If $f(\mbf{x}) = \sum_{\Z^{n}}\; a_{\mbf{d}}\mbf{x}^{\mbf{d}}$ is a any polynomial over a local field $\res$ with a valuation $w: \res \rightarrow \Z,$ in \cite{EKL} the \textit{extended Newton polytope} of $f$ was defined as follows:
\[
N_{1}(f) = \text{Conv}\{ (\mbf{d}, u): \mbf{d} \in \text{Supp}(f) \text{ and } u \geq w(a_{\mbf{d}}) \} 
\]
Now suppose $f(\mbf{x})$ is a polynomial (with the same explicit form as above) with coefficients in $\K,$ a \ndim{n} local field. We define the \textit{n- extended Newton polytope} of $f$ as follows:
\[
N_{n}(f) \eqdef \text{ Conv} \{ (\mbf{d}, j_{1},\ldots, j_{n}) :  \mbf{d} \in \text{ Supp}(f) ; \; j_{s} \geq \nu_{s}(a_{\mbf{d}}), 1 \leq s \leq n \} 
\]\label{NP}
Here $\nu : \no{\K^{al}} \rightarrow \Gamma_{\R}$  is the map constructed before and $\nu_{s}$ is the $s -$ coordinate of $\nu.$

 \begin{thm} \label{T:Newton}
 Suppose $f(t)$ is a polynomial in one variable over $\K$ given by
 \[
 f(t)  = 1+ \sum_{i=1} ^{n} a_{i} t^{i} .
 \]  Let,
\begin{eqnarray} \label{eq:E1}
 f(t) = \prod_{i=1}^{n} (1- t/ \alpha_{i}) \; \mbox{with possibly repeated roots in } \K^{al}.
 \end{eqnarray}
 Let $\nu_{1} > \nu_{2} > \ldots > \nu_{m}$ be $m$ points in $\Gamma_{\R}$ with the decreasing order, such that exactly $k_{1}$ roots have valuation $\nu_{1},$ $k_{2}$ roots have valuation $\nu_{2}$ and so on (i.e. $k_{1} + k_{2} + \ldots+ k_{m} = n$). \\
 Then $N_{n}(f)$ has exactly $m$ bounded edges with generalized slopes $- \nu_{i}.$ 
 \end{thm}

  \begin{proof}
Without loss of generality, we will assume that in the factorization (\ref{eq:E1}), the roots $\alpha_{1}, \ldots, \alpha_{k_{1}}$ have valuation $\nu_{1},$  the roots $\alpha_{k_{1}+1}, \ldots, \alpha_{k_{1}+k_{2}}$ have valuation $\nu_{2}$ and so on. \\
We have  $a_{i} = (-1)^{i} \mbf{e}_{i}(1/ \alpha_{1}, \ldots, 1/ \alpha_{n}),$ where $\mbf{e}_{i}$ is the $i$-th elementary symmetric function on $n$ variables.

Applying valuation, $\nu: \no{\K}\rightarrow \Gamma_{\R},$ we get

\begin{equation}\label{eq:E2}
\nu(a_{i}) = \nu(\mbf{e}_{i}(1/\alpha_{1}, \ldots, 1/\alpha_{n}))  \geq \min \left \{ \sum_{\vert J\vert = i} (\nu({1/\alpha_{j_{1}}}) + \ldots \nu({1/\alpha_{j_{i}}}))  \right \} 
\end{equation}

where $J = \lbrace j_{1}, \ldots, j_{i} \rbrace $ runs over subsets of $\lbrace 1, 2, \ldots, n\rbrace$ of cardinality $i.$

When $i=0,$  $a_{0}= 1$ and $\mbf{e}_{0} = 1$ by convention.

When $0 < i < k_{1} ,$ the minimum in the equation (\ref{eq:E2}), is attained at more than one index set $ J .$ 
So, the inequality is strict, i.e. $\nu(a_{i})  > \min \{ \sum_{\vert J\vert = i} (\nu({1/\alpha_{j_{1}}}) + \ldots \nu({1/\alpha_{j_{i}}}))  \}. $ When $ i = k_{1}$ the minimum is unique and thus we have $\nu(a_{k_{1}}) =  -k_{1}\nu_{1}.$ We thus conclude that a (possibly) bounded edge emanating from the vertex $(0,0)$ of $N_{n}(f)$ passes through the vertex $(k_{1}, \nu(a_{k_{1}})).$  This edge has generalized slope $-\nu_{1}.$\\
Next, when $k_{1} < i ,$ equation (\ref{eq:E2}) and monotonicity of $v_{i}$ imply that any vector joining $(0,0) $ to $(i, \nu(a_{i}))$ has generalized slope $ >  -\nu_{1} .$ Thus we conclude that exactly one bounded edge emanates from $(0,0)$ and terminates at $(k_{1}, \nu(a_{k_{1}})).$ The points $ (j,\nu(a_{j}))$ for $ 0 < j < k_{1} $ lie inside the polyhedron. So, the next possibly bounded edge emanates from the vertex $(k_{1}, a_{k_{1}}).$ \\
Repeating, the previous argument with $ k_{1}< i < k_{1}+k_{2} $; using equation (\ref{eq:E2}) and the monotonicity of  $\nu_{i},$ we conclude that exactly one bounded edge emanates from $(k_{1}, \nu(a_{k_{1}}))$ and terminates at $(k_{1}+k_{2}, \nu(a_{k_{1}+k_{2}}))$ with generalized slope $-\nu_{2}.$  \\
Continuing like this, we see that there is exactly one edge emanating from $(k_{1} + \ldots + k_{j}, \nu(a_{k_{1}+ \ldots+ k_{j}}))$ which terminates at $((k_{1} + \ldots + k_{j+1}, \nu(a_{k_{1}+ \ldots+ k_{j+1}})))$ with generalized slope $ - \nu_{j}.$
\end{proof}

\begin{rem}Let us consider $\displaystyle f(t_{1}, \ldots, t_{n}, x)$  in the ring $\displaystyle \Kn{0}[t_{1}, \ldots, t_{n}, x].$ We are interested in (fractional) power series $\displaystyle \phi(t_{1}, \ldots, t_{n})$ such that formally, it satisfies the identity $\displaystyle f(t_{1}, \ldots, t_{n}, \phi) = 0.$

The case $n=1$ was already considered by Newton.
For $n > 1,$ an algorithm to construct such a power series solution has been given by McDonald. His construction gives explicit description of the co-efficients arising in $\phi,$ in terms of the Newton polytope of $f$ and a choice of an (admissible) edge of the polytope. Moreover he shows that complete systems of series solutions are parametrized by coherent edge paths of the Newton polytope in $\R^{n+1}. $ We refer the reader to \cite{McD} for the details of this construction. 

In this framework, we are only interested one particular series solution corresponding to the edge-path determined by the valuation $\nu$ on the coefficients of $f.$
\end{rem}

\section{Tropicalization}
Let $\tor_{\K} \eqdef \mathbb{G}_{m, \K}^{m}$ be the \ndim{m} algebraic torus over $\K,$ where $\K$ is a \ndim{n} local field. Let $\tor_{\K} (\Lf)$ denote the $\Lf$ valued points. Let $M$ denote the lattice of characters of $\tor_{\K} .$
  
By definition, we have a non-degenerate pairing $ ev: \tor_{\K} (\K^{al}) \times M \longrightarrow \no{\K^{al}}$ which is given by evaluation of a character at a point ($ ev(\mbf{t}, \chi) \mapsto \chi(\mbf{t})$). 
This defines a map $ \displaystyle \tor_{\K}(\K^{al}) \longrightarrow \mbox{Hom}(M, \K^{al}).$ 
Composing with the valuation $\nu: \K^{al} \rightarrow \Gamma_{\R}$ gives us a map $\rho :\tor_{\K}(\K^{al})  \longrightarrow \mbox{Hom}(M, \Gamma_{\R}).$
\begin{defn} \label{D:trop}
Tropicalization is the vertical arrow which makes the following diagram commutative. The diagonal arrow is just pointwise valuation.

\begin{displaymath}
\xymatrix{
\tor_{\K}(\K^{al}) \ar[d]_{Trop} \ar[rd]^{\nu}  &  \\
\text{Hom}(M, \Gamma_{\R})  \ar[r]_{ev} & \Gamma_{\R}^{m} }
\end{displaymath}

\end{defn}

\begin{rem}
\ben[(i)]
\item Tropicalization is the map $\rho$ constructed above with some additional normalization conditions.
\item If $\X$ is any closed, reduced sub-scheme of the torus, then $Trop(\X)$ is defined as the topological closure of the image. Clearly $Trop(\X) = Trop(\X_{1}) \cup Trop(\X_{2})$ if  $\X_{i} $ is an irreducible component, so we may as well assume that $\X$ is irreducible.
\een
\end{rem}

\subsection{Hypersurface in a torus}
Let $f(\mbf{x})$ be a Laurent polynomial over $\K$ defining a closed, reduced, irreducible subscheme $\X_{f}$  of  the torus $\tor_{\K}.$ \\
Let us assume, to be specific, that $f$ is given explicitly  as 
\[
f(\mbf{x}) =  \sum_{\mbf{d} \in \Z^{m}} a_{\mbf{d}} \mbf{x}^{\mbf{d}}
\]

\noindent Associated to $f$ we define a piecewise linear map $ f^{\tau} : \Gamma_{\R}^{m} \rightarrow \Gamma_{\R}$ defined by 
\begin{eqnarray*}
f^{\tau}(u_{1}, \ldots, u_{m}) & = & \min_{\text{Supp}(f)}\left \{ \nu(a_{\mbf{d}}) + \sum d_{i}. u_{i}  \right \} 
\end{eqnarray*}
where $u_{i} \in \Gamma_{\R}$ and $d_{i}. u_{i} $ is the usual scalar multiplication in $\Gamma_{\R}.$

Let $T(f) \eqdef $ the locus of points in $\Gamma_{\R}$ where the minimum is attained at two or more distinct indices. This is equivalent to saying that  is the $T(f)$ is the non-differentiability locus of $f^{\tau}.$ Evidently $T(f)$ is a $\Gamma$ rational polyhedron in $\Gamma_{\R}^{m}$ of dimension $nm-n.$

\begin{thm} \label{T:amoeba}
The subsets $T(f)$ and $ \mathcal{A}(\X_{f}) \eqdef \overline{\{ \nu(\mbf{t}) : \mbf{t} \in \X_{f}(\K^{al}) \}}$ of $\Gamma_{\R}^{m}$ are equal.
\end{thm}

\begin{proof}
It suffices to show that $T(f) \cap \Gamma^{m}$ = $ \mathcal{A}(\X_{f}) \cap \Gamma^{m} $ where $\Gamma^{m}$ is the canonical divisible subgroup of $\Gamma_{\R}^{m}.$ \\
If $ c \in \mathcal{A}(\X_{f}) \cap \Gamma^{m}$ then there is a $\mbf{t} \in (\no{\K^{al}})^{m}$ such that $\nu(\mbf{t}) = c.$ As a result $f(\mbf{t}) = 0$ which forces $f^{\tau}$ to be non-differentiable at $c.$ This shows that 
$ \mathcal{A}(\X_{f}) \cap \Gamma^{m} \subset T(f) \cap \Gamma^{m} .$ \\
Now suppose $c \in T(f) \cap \Gamma^{m}$  then there is a $\mbf{t} \in (\no{\K^{al}})^{m}$  such that $v(\mbf{t}) =  c .$ We will be done if we show that $f(\mbf{t}) = 0.$ After a change of co-ordinates we reduce to the case when $c = 0.$

\begin{lema}
If $\mbf{0} \in T(f) \cap \Gamma^{m}$ then $\mbf{0} \in \mathcal{A}(\X_{f}) \cap \Gamma^{m}.$
\end{lema}
\begin{proof}
$\mbf{0} \in T(f) \cap \Gamma^{m}$ implies that there are $k$ indices ($k \geq 2$) $\mbf{d}_{1}, \ldots, \mbf{d}_{k}$  such that 
\begin{eqnarray}  \label{E:cond}
\nu(a_{\mbf{d}_{1}}) =\nu(a_{\mbf{d}_{2}}) = \ldots = \nu(a_{\mbf{d}_{k}})  < \nu(a_{\mbf{d}_{k^{\prime}}}) 
\end{eqnarray}
for every $\mbf{d}_{k^{\prime}}  \in \text{ Supp}(f) \setminus \{\mbf{d}_{1}, \ldots, \mbf{d}_{k} \}.$ Factoring out $a_{\mbf{d}_{1}}$ if necessary we may assume that $\nu(a_{\mbf{d}_{1}}) = 0.$ For a generic $\mbf{b} \in \Z^{m}$ such that the usual inner-product $\mbf{b}. \mbf{d}_{1} \neq \mbf{b}. \mbf{d}_{2} = \ldots \neq \mbf{b}. \mbf{d}_{k}$ consider the polynomial $f_{\mbf{b}}(z) \eqdef \sum_{\mbf{d}\in \Z^{m}} a_{\mbf{d}} \;z^{\mbf{b}.\mbf{d}} .$ Then by condition (\ref{E:cond}), $N_{n}(f)$ will have bounded edges with generalized slope $\mbf{0}.$ Thus we have a root $p$ of $f_{\mbf{b}}(z)$ such that $\nu(p)= 0.$ Clearly $\mbf{p}_{0} = (p^{b_{1}}, \ldots, p^{b_{m}})$  is a root of the required form.
\end{proof}
This completes the proof of theorem (\ref{T:amoeba}).
\end{proof}

\begin{cor}\label{theorem: hyp}
The tropicalization of a reduced, irreducible hypersurface $\X_{f}$  in $\tor_{\K},$ is a rational polyhedral complex of dimension $mn-n.$
\end{cor}
\subsection{General closed sub-schemes}
We now consider the  general case where $\X$ is a closed, reduced, irreducible $\K$ sub-scheme of $\tor_{\K}$ of dimension $d.$ 
Let $\I_{\X}$ denote the ideal defining $\X.$ 
\subsubsection{ Initial ideals and Weight functions}

Let $M$  be an abelian group, following analogy with the character lattice we will use $\chi^{u}$ to denote an element of $M$  and $\chi^{u+v} = \text{  the sum of } \chi^{u} \text{ and } \chi^{v}.$ Let $\K[M]$ denote the group ring over a field $\K.$ Let $\Delta^{\times} \subset \K[M]$ denote the group of monomials (i.e. elements of the form 
$ a\chi^{u}$ where $a \neq 0$). Let $(G, +)$ be any ordered abelian group. Generalizing the approach of Sturmfels \cite{Sturm}, we consider the following definition.

\begin{defn}
A weight function on $\K[M]$ is any group homomorphism \linebreak[4] $\displaystyle \text{ \textbf{wt}}: \Delta^{\times} \rightarrow G.$ We extend \textbf{wt} to non-zero $ f \eqdef \sum_{u \in \Delta^{\times}} a \chi^{u}$ by defining  
\[\text{ \textbf{wt}}(f) = \min_{u \in \text{ Supp}(f)} \left \{ \text{\textbf{wt}}(a \chi^{u})\right \}. \]
\end{defn} 
\begin{exmple}\label{E:sturm} If $M = \Z^{m}$, then $\K[M] \cong \K[x_{1}^{\pm}, \ldots, x_{m}^{\pm}],$ and $G = (\R, +),$ the set $\text{Hom}(M, \R) $ provides a natural space of weight functions ($\text{\textbf{wt}}_{\omega}(a \chi^{u}) = \omega(u)$ for $\omega \in \text{Hom}(M, \R)$). With this identification,  $\R^{n}$ appears as the space of weight vectors in \cite[Chapter 1]{Sturm}.

The above definition easily adapts to the case of rings like $\K[x_{1}, \ldots, x_{m}]$ by using monoids ($M = \mathbb{N}^{m}$)  instead of groups.
\end{exmple}

We will consider the case when, $M$ is a character lattice of the torus, $\K$ is a \ndim{n} local field and $G = \Gamma_{\R}.$ For a  $\omega \in \mbox{Hom}(M, \Gamma_{\R})$ we will consider the following weight function (and its extension to linear combination of monomials as defined above).
\begin{equation} \label{eq:weight}
 \mbox{\textbf{wt}}_{\omega}\left (a_{u}\chi^{u} \right) = \nu(a_{u}) - \langle \omega, u \rangle
\end{equation}

Given an weight function \textbf{wt} on $\K[M]$ and a non-zero $f \in \K[M]$ let  $In_{\text{\textbf{wt}}}(f) $ be the linear combination of the terms of lowest weight in $f.$ More generally if $J$ is an ideal of $\K[M],$ then $In_{\text{\textbf{wt}}}(J)$ is the ideal generated by the set $\displaystyle \left \{ In_{\text{\textbf{wt}}}(f) : f \in J \right \}.$

\begin{rem}
\ben[(i)]
\item In the setup of example (\ref{E:sturm}), $In_{\text{\textbf{wt}}_\omega}(f)$  and $In_{\text{\textbf{wt}}_\omega}(J)$ are called (in \cite{Sturm}), \textit{the initial form} and \textit{the initial ideal} respectively. We will keep this terminology in our case as well.
\item We will for simplicity, use $In_{\omega}$ to denote $In_{\text{\textbf{wt}}_\omega}$ where the weight function will always be defined by (\ref{eq:weight}).
\een
\end{rem}

Recall that $M$ denotes the character lattice of $\tor_{\K}.$  So the ring of functions on the torus $\K^{al}[ \tor_{\K}]$ is the group ring $\K^{al}[M].$ Let us fix $\omega \in \text{Hom}(M, \Gamma_{\R}).$ Then by equation (\ref{eq:weight}) we get a weight function $\text{\textbf{wt}}_{\omega}$ on  $\K^{al}[M].$ \\
Let  $\Rint{\nu}[M]^{\omega} \eqdef \{ f \in \K^{al}[M] \text{ such that}\; \mathbf{wt}_{\omega}(f) \geq 0 \}.$ These are analogs of the ``tilted group rings" considered by Payne in \cite{Pay1}.  Then $\tor^{\omega}  \eqdef \mbox{Spec}(\Rint{\nu}[M]^{\omega})$ defines an integral model of $\tor_{\K}$ over $\mbox{ Spec}(\Rint{v}).$ 
Let $\X^{\omega}$ be the Zariski closure of $\X$ in $\tor^{\omega}$ and $\X_{\omega}$ denote the fiber of $\X^{\omega}$ over $\text{ Spec}(\Rint{\nu}/ \mathfrak{m}_{\nu}),$ where $\mathfrak{m}_{\nu}$ is the maximal valuation ideal of $\Rint{\nu}.$

Let us now introduce a system of co-ordinates on $\Gamma_{\R}$ (recall $\Gamma_{\R} \cong \R^{n}$) and let $\omega = \bfvec{\omega}$ in co-ordinates. We now use $\omega_{n}$ to define a different weight function on $\K^{al}[M]$ which only considers the $\K$ as a local field. More precisely:
\begin{eqnarray*}
\text{\textbf{wt}}_{\omega_{n}}\left (\sum_{u \in M} a_{u}\chi^{u}\right) &=& \sum_{u} w_{(n)}(a_{u}) - \langle \omega_{n}, u \rangle .
\end{eqnarray*}

As before, we can consider $\Rint{\Kn{n}}[M]^{\omega_{n}}$ as the tilted group ring and \linebreak[3] $\tor^{\omega_{n}} \eqdef \mbox{Spec}(\Rint{\Kn{n}}[M]^{\omega_{n}})$ defines an integral model of the torus $\tor_{\K}.$

Now consider $\X^{\omega_{n}}$ the Zariski closure of $\X$ in $\tor^{\omega_{n}} $ and let $\X_{\omega_{n}}$ denote the fiber of $\X^{\omega_{n}}$  over the unique maximal ideal $\mideal{\Kn{n}}$ of $\Rint{\Kn{n}}.$ Clearly $\X_{\omega_{n}}$ is a closed subscheme of the torus $\tor_{\Kn{n-1}},$ the $m-$ dimensional torus defined over $\Kn{n-1}.$ Now consider a weight function on $\Kn{n-1}^{al}[\tor_{\Kn{n-1}}]$ given by $\omega_{n-1}$ and repeating the above steps word-for-word we obtain a sub-scheme $\X_{(\omega_{n-1}, \omega_{n})}$ of $\tor_{\Kn{n-2}}.$

Iterating  this process $n$ times we have a closed sub-scheme $\X_{\bfvec{\omega}}$ of $\tor_{\Kn{0}}.$\\
As closed subschemes of the torus $\tor_{\Kn{0}}$ over $\Kn{0}$ we have an isomorphism $ \displaystyle 
\X_{\omega}\cong \X_{\bfvec{\omega}}$ where $\displaystyle \omega$ corresponds to $\bfvec{\omega}$ in co-ordinates.

\begin{thm}
$\displaystyle Trop(\X)  = \left \{ \omega \in \mbox{ Hom}(M, \Gamma_{\R})  \mbox{ such that } \X_{\omega} \neq \emptyset \right \} $
\end{thm}

\begin{proof}
If $\omega \in Trop(\X) $ then there is $\mbf{t} \in \tor_{\K^{al}}$ such that $f(\mbf{t}) =  0 $ for every $f \in \I_{\X}.$ As a result $In_{\omega}(f)$ is not monomial. Thus the ideal $In_{\omega}(\I_{\X})$ is not monomial and hence it defines a non-empty subscheme of the torus (i.e. there are solutions to these polynomials with all co-ordinates non-zero). Thus proves that $\X_{\omega}$ is not empty.

Recall that  $\X_{\omega} = \X_{\bfvec{\omega}}.$ We will be done if we can show that any point of $\X_{\omega}(\Kn{0}^{al})$ lifts to a point of $\X(\K^{al}).$
    
  \begin{lema}
  Using the above notation, any point of $\X_{\omega}(\Kn{0}^{al})$ lifts to a point of $\X(\K^{al}).$  
  \end{lema}
  
  \begin{proof}
  We consider the following chain of schemes 
$ \displaystyle
\X \rightsquigarrow \X_{\omega_{n}} \rightsquigarrow  \X_{(\omega_{n-1},\omega_{n})}  \rightsquigarrow \ldots \X_{\bfvec{\omega}} $;  where  $\X_{(\omega_{i}, \ldots, \omega_{n})}$ is  a special fiber (over the maximal valuation ideal of the DVR  \Rint{\Kn{i }^{al}}) of an integral model of $\X_{(\omega_{i+1}, \ldots, \omega_{n})}.$

 It follows from \cite[Theorem 4.1]{Pay1} that any point on $\X_{(\omega_{i+1}, \ldots, \omega_{n}})$ lifts to a point of $\X_{(\omega_{i}, \ldots, \omega_{n})}.$ \footnote{In fact, Payne proves that the fibers are Zariski dense.} Thus we  have the desired lift of a point of $\X_{\bfvec{\omega}}.$
 \end{proof}
 This finishes the proof of the theorem.
 \end{proof}
 
  \begin{exmple}
  We will now consider the polynomial $f(x,y) = x + y + 1,$ and describe $Trop(\X_{f})$ over a two dimensional local field $\C((t_{1}))((t_{2}))$. 
 
 To do this, we consider weights as vectors in $\R^{4}$ of the form $(\omega_{11}, \omega_{12}; \omega_{21}, \omega_{22}).$ For such a weight consider the polynomial  \[f(t_{1}^{\omega_{11}}t_{2}^{\omega_{12}}x, t_{1}^{\omega_{21}}t_{2}^{\omega_{22}}y) = t_{1}^{\omega_{11}}t_{2}^{\omega_{12}}x + t_{1}^{\omega_{21}}t_{2}^{\omega_{22}}y + 1. \]
 There are now two steps. We first consider initial terms with respect to the $t_{2}$ variables and then consider degeneration with respect to $t_{1}$ variables as summarized in table \ref{T: init}. As a result the tropicalization in this case, is a two dimensional polyhedral complex embedded in $\R^{4}.$ It ``looks" like figure \ref{F:two}.
  \begin{table}[h] \label{T: init}
 \caption{\small{Tropicalization of $x + y + 1 = 0$ over $\C((t_{1}))((t_{2})).$ }}
 {\footnotesize
\begin{tabular}{| c |  c || c | c |}
 \hline 
 Condition   & Initial  & Condition  & Initial  \\
 on weights & Degenerations & on weights & Degenerations \\
 $\omega_{12} , \omega_{22}$ &  & $\omega_{12} , \omega_{22}$ &\\
 \hline \hline
 $\omega_{12} = \omega_{22} < 0 $  & $ t_{1}^{\omega_{11}}x +t_{1}^{\omega_{21}} y$  &  $\omega_{11} = \omega_{21}$ & $x + y$ \\
 \hline 
  $\omega_{12} = \omega_{22} = 0 $  & $ t_{1}^{\omega_{11}}x +t_{1}^{\omega_{21}} y + 1$  & 
  $\omega_{11} = \omega_{21} <  0$ & $x + y$ \\
  \hline
 ditto  & ditto & $\omega_{11} = \omega_{21} = 0 $ & $ x + y + 1$ \\
  \hline
 ditto & ditto &  $\omega_{11} = 0  ; \omega_{21} > 0 $  & $x + 1$ \\
  \hline
  ditto & ditto & $\omega_{11}> 0  ; \omega_{21} = 0 $  & $y + 1$ \\
  \hline 
  $\omega_{12} = 0 ; \omega_{22} > 0 $  & $ t_{1}^{\omega_{11}}x + 1$ & $\omega_{11} = 0 ; \omega_{21} \in \R$ & $x+ 1$  \\
  \hline
  $\omega_{12}> 0 ; \omega_{22} = 0 $  & $t_{1}^{\omega_{21}}y + 1$ & $\omega_{11}  \in \R ; \omega_{21} = 0 $ & $ y + 1$  \\
  \hline
  \end{tabular}
}
\end{table}
 
 \begin{figure}[h]
 \begin{center}
  \includegraphics[scale = 0.75]{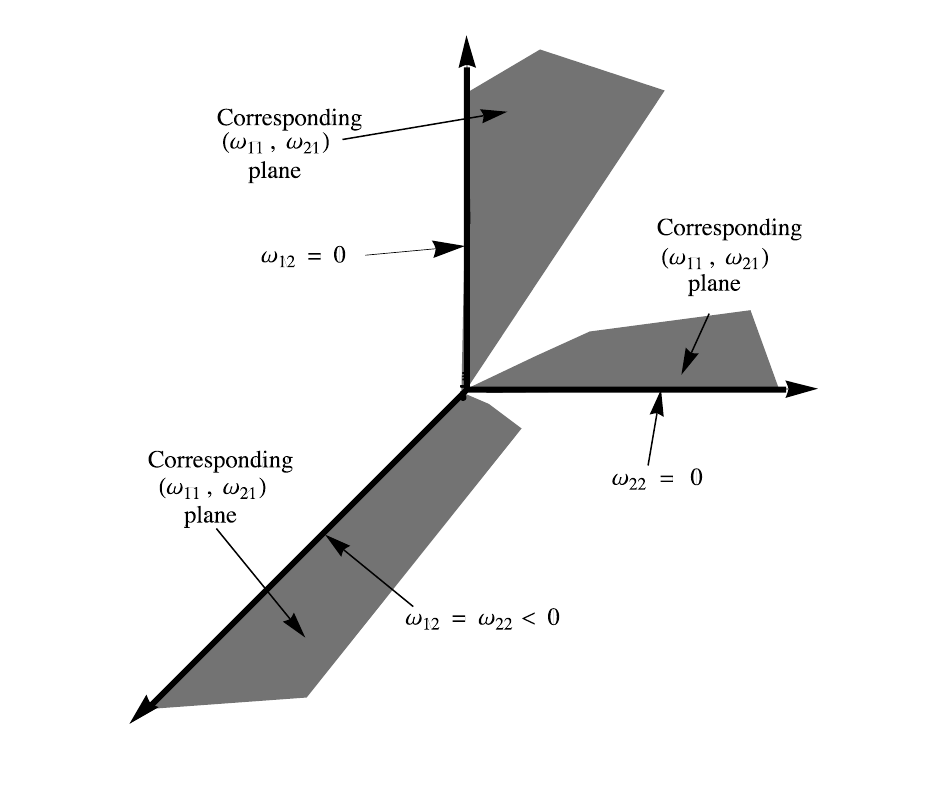}
  \end{center}
  \caption{\small {Tropicalization of $ x +  y +1$ over $\C((t_{1}))((t_{2})).$}}
  \label{F:two}
  \end{figure}
   We recall that the tropicalization of $\X_{f}$ over a local field, say $\C((t))$ comprises of three half lines meeting at the origin, see \cite{Gath} for a nice exposition.
  \end{exmple}
 \subsubsection{Dimension}
It is well known from the work of \cite{Bieri-Groves, EKL, Ss} that if $\X$ is a $d$- dimensional subscheme of the torus, over a local field then $Trop(X)$ is a rational polyhedral complex of pure dimension $d.$  

In the case of an \ndim{n} local field, let us start with the simple observation that if $\K$ is an \ndim{n} local field over $\Kn{0},$ then it is also a \ndim{j} local field over $\Kn{n-j}.$ 

So, if $\X$ is a scheme defined over $\K$ (we always assume $\X$ is a closed, reduced, irreducible sub-scheme of $\tor_{\K}$) then we let $\mathfrak{Trop}^{(j)}(\X)$ denote the tropicalization  of $\X$ over $\K,$ where $\K$ is now considered as a \ndim{j} local field over $\Kn{n-j}.$ In this notation, we are interested in $Trop(\X)  = \mathfrak{Trop}^{(n)}(\X).$
Consider the retraction maps of ambient vector spaces $\mathfrak{r}^{j}_{n}: \mathfrak{Trop}^{(n)}(\X) \rightarrow \mathfrak{Trop}^{(j)}(\X) .$ The fiber of $\mathfrak{r}^{j}_{n}$ over any point $(\omega_{j}, \ldots, \omega_{n}) \in \mathfrak{Trop}^{(j)}(\X)$ is $Trop(\X_{(\omega_{j}, \ldots, \omega_{n})}).$ 

The dimension of $\X_{(\omega_{j}, \ldots, \omega_{n})}$ is that of dim($\X$) for generic choice of $(\omega_{j}, \ldots, \omega_{n})$ (and it can drop for special values of $\omega_{j}$), thus by induction on the dimension of the local field $\K$ we get that dim($Trop(\X)) = nd .$
\begin{thm} \label{T:main2}
The tropicalization of a \ndim{d} closed, reduced, irreducible subscheme of an algebraic torus over a \ndim{n} local field is a rational polyhedral complex of dimension $nd.$
\end{thm} 

We finally consider an example showing that the polyhedral complex associated to a tropicalization is not necessarily pure. Recall, a polyhedral complex  of dimension $d$ is called pure if all the maximal faces are of the same dimension $d.$ 
\begin{exmple} \label{E:impure}
Let us consider the polynomial $f(x, y) = (x- t_{1}) (x - t^{2}_{1}) + y^{2},$ defined over $\C((t_{1}))((t_{2})).$  For any given weight $(\omega_{11}, \omega_{12}; \omega_{21}, \omega_{22}) \in \R^{4},$ we consider the polynomial $f(t_{1}^{\omega_{11}}t_{2}^{\omega_{12}}x, t_{1}^{\omega_{21}}t_{2}^{\omega_{22}}y).$ Consider the table \ref{T:imex} describing the relevant initial degenerations of $f.$  

\begin{table}[h]\label{T:imex}
\caption{\small{Tropicalization of $(x- t_{1}) (x - t^{2}_{1}) + y^{2} = 0$ over $\C((t_{1}))((t_{2})).$ }}
{\footnotesize
\begin{tabular}{|l|l||l|l|}
 \hline
 Condition & Initial & Condition  & Initial  \\
on weights & Degenerations &on weights & Degenerations \\
$\omega_{12} , \omega_{22}$ & &  $\omega_{12} , \omega_{22}$ & \\
 \hline \hline
  $\omega_{12} = \omega_{22} < 0 $  & $  t_{11}^{2\omega_{11}}x^{2} + t_{1}^{2\omega_{21}} y^{2}$  &  $\omega_{11} = \omega_{21} $ & $x^{2} + y^{2}$\\

 \hline 
  $\omega_{12} = \omega_{22} = 0 $  & $(t_{1}^{\omega_{11}}x - t_{1})(t_{1}^{\omega_{11}}x - t_{1}^{2})  + t_{1}^{2 \omega_{21}} y^{2}$  & 
  $\omega_{11} = \omega_{21} <  1 $ & $x^{2} + y^{2}$ \\
  \hline
ditto  & ditto & $\omega_{11} =  \omega_{21} = 1 $  & $ x^{2} - x + y^{2}$ \\
  \hline
 ditto & ditto &  $\omega_{11} = 1 ; \omega_{21} > 1 $  & $x^{2} - x$ \\
  \hline
  ditto & ditto & $ 1< \omega_{11} < 2 ; \omega_{21} = \omega_{11} / 2 $  & $ - x + y^{2}$ \\
  \hline
   ditto & ditto & $ \omega_{11} =2 ; \omega_{21} \in \R $  & $ x - 1$ \\
  \hline 
  $\omega_{12} = 0 ; \omega_{22} > 0 $  & $ (t_{1}^{\omega_{11}}x - t_{1})(t_{1}^{\omega_{11}}x - t_{1}^{2})$ & $\omega_{11} = 1 \text{ or } 2 ; \omega_{21} \in \R$ & $ x - 1$  \\
  \hline
 \end{tabular}
 }
\end{table}

In this example the line segment given by $(t, 0, t/2, 0)$ where $1 \leq t \leq 2$, joining the points $(1,0,1/2,0)$ to $(2,0,1,0)$ is a maximal face of dimension 1. So the polyhedral complex $Trop(\X_{f})$ is not pure.
\end{exmple}


\end{document}